\documentclass[english,12pt,a4paper,twoside]{smfart}
\usepackage{amsfonts,smfthm}
\usepackage{amsmath}
\usepackage{stmaryrd}
\usepackage{amssymb}
\usepackage{graphics}
\usepackage{babel}
\usepackage{enumerate}
\usepackage{eufrak}
\usepackage{euscript}

\usepackage{newlfont}
\usepackage{latexsym}
\usepackage{graphicx}
\usepackage{a4wide}

\usepackage{cite}        
\usepackage{url}         

\title{A New Upper Bound for the Cross Number \\ of Finite Abelian Groups }
\author{Benjamin Girard}
\thanks{B. GIRARD, Centre de Math\'{e}matiques Laurent Schwartz, UMR $7640$ du CNRS,
\'Ecole polytechnique, $91128$ Palaiseau cedex, France
(\textit{e-mail}: benjamin.girard@math.polytechnique.fr)}

\address{Centre de
Math\'{e}matiques Laurent Schwartz, UMR $7640$ du CNRS, \'Ecole
polytechnique, $91128$ Palaiseau cedex, France.}
\email{benjamin.girard@math.polytechnique.fr}

\theoremstyle{plain}
\newtheorem{theorem}{Theorem}[section]

\newtheorem{lem}[theorem]{Lemma}
\newtheorem{corollary}[theorem]{Corollary}
\newtheorem{conjecture}[theorem]{Conjecture}

\theoremstyle{definition}

\theoremstyle{remark}

\def\cnp[#1,#2]{\begin{pmatrix} #1 \\#2 \end{pmatrix}}
\begin{document}

\maketitle \setcounter{page}{1} \vspace{0.0cm}

\begin{abstract} In this paper, building among others on earlier works by U.~Krause
and C.~Zahlten (dealing with the case of cyclic groups), we obtain a
new upper bound for the little cross number valid in the general
case of arbitrary finite Abelian groups. Given a finite Abelian
group, this upper bound appears to depend only on the rank and on
the number of distinct prime divisors of the exponent. The main
theorem of this paper allows us, among other consequences, to prove
that a classical conjecture concerning the cross and little cross
numbers of finite Abelian groups holds asymptotically in at least
two different directions.
\end{abstract}

\vspace{-0.7cm}
\section{Introduction}
Let $G$ be a finite Abelian group, written additively. By
$\text{r}(G)$ and $\exp(G)$ we denote respectively the rank and the
exponent of $G$. If $G$ is cyclic of order $n$, it will be denoted
by $C_n$. In the general case, we can decompose $G$ (see for
instance \cite{Samuel}) as a direct product of cyclic groups
$C_{n_1} \oplus \dots \oplus C_{n_r}$ where $1 < n_1 \text{ }
|\text{ } \dots \text{ }|\text{ } n_r \in \mathbb{N}$, so that every
element $g$ of $G$ can be written $g=[a_1,\dots,a_r]$ (this notation
will be used freely along this paper), with $a_i \in C_{n_i}$ for
all $i \in \llbracket 1,r \rrbracket=\{1,\dots,r\}$.

\medskip
In this paper, any finite sequence $S=\left(g_1,\dots,g_l\right)$ of
$l$ elements from $G$ will be called a \em sequence \em of $G$ with
\em length \em $l$. Given a sequence $S=\left(g_1,\dots,g_l\right)$ of $G$, we say that
$s \in G$ is a \em subsum \em of $S$ when it lies in the following
set, called the \em sumset \em of $S$:
$$\Sigma(S)=\left\{ \displaystyle\sum_{i \in I} g_i \text{ }|\text{ } \emptyset \varsubsetneq I
\subseteq \{1,\dots,l\}\right\}.$$

\noindent If $0$ is not a subsum of $S$, we say that $S$ is a \em
zero-sumfree sequence\em. If $\sum^l_{i=1} g_i=0,$ then $S$ is said
to be a \em zero-sum sequence\em. If moreover one has $\sum_{i \in
I} g_i \neq 0$ for all proper subsets $\emptyset \subsetneq I
\subsetneq \{1,\dots,l\}$, $S$ is called a \em minimal zero-sum
sequence\em.

\medskip
In a finite Abelian group $G$, the order of an element $g$ will be
written $\text{ord}(g)$ and for every divisor $d$ of the exponent of
$G$, we denote by $G_d$ the subgroup of $G$ consisting of all the
elements of order dividing $d$: $$G_d=\left\{x \in G \text{ } |
\text{ } dx=0\right\}.$$ In a sequence $S$ of elements of $G$, we
denote by $S_d$ the subsequence of $S$ consisting of all the
elements of order $d$ contained in $S$.

\medskip
Let $\mathcal{P}=\left\{p_1=2<p_2=3<\dots\right\}$ be the set of
prime numbers. Given a positive integer $n \in
\mathbb{N}^*=\mathbb{N}\backslash\{0\}$, we denote by
$\mathcal{D}_n$ the set of its positive divisors. If $n
> 1$, we denote by $P^-(n)$ the smallest prime element of
$\mathcal{D}_n$, and we put by convention $P^-(1)=1$. By $\tau(n)$
and $\omega(n)$ we denote respectively the number of positive
divisors of $n$ and the number of distinct prime divisors of $n$.

\medskip By $\mathsf{D}(G)$ we denote
the smallest integer $t \in \mathbb{N}^*$ such that every sequence
$S$ of $G$ with length $|S| \geq t$ contains a zero-sum subsequence.
The constant $\mathsf{D}(G)$ is called the \em Davenport constant
\em of the group $G$.

\medskip
By $\eta(G)$ we denote the smallest integer $t \in \mathbb{N}^*$
such that every sequence $S$ of $G$ with length $|S| \geq t$
contains a zero-sum subsequence $S' \subseteq S$ with length $|S'|
\leq \exp(G)$. Such a subsequence is called a \em short zero-sum
subsequence\em.

\medskip
The constants $\mathsf{D}(.)$ and $\eta(.)$ have been extensively
studied during last decades and even if numerous results were proved
(see Chapter 5 of the book \cite{GeroKoch05} or \cite{GaoGero06} for
a survey and many references on the subject), their exact values are
known for very special types of groups only. In the sequel, we shall
need some results on some of the groups for which we know the exact
values, so we gather what is known concerning them in the following
theorem.

\begin{theorem} \label{Proprietes D et Eta} The two following statements hold:
\begin{itemize}
\item[$(i)$] Let $p \in \mathcal{P}$, $r \in \mathbb{N}^*$ and
$ \alpha_1 \leq \dots \leq \alpha_r$, where $\alpha_i \in
\mathbb{N}^*$ for all $i \in \llbracket 1,r \rrbracket$. Then, for
the $p$-group $G \simeq C_{p^{\alpha_1}} \oplus \cdots \oplus
C_{p^{\alpha_r}}$, we have:
$$\mathsf{D}(G)=\displaystyle\sum^r_{i=1} \left(p^{\alpha_i}-1\right)+1.$$

\item[$(ii)$] For every $m,n \in \mathbb{N}^*$ with $m | n$, we have:
$$\mathsf{D}(C_m \oplus C_n)=m+n-1 \hspace{0.23cm} \text{ and }
\hspace{0.31cm} \eta(C_m \oplus C_n)=2m+n-2.$$ In particular, we
have $\mathsf{D}(C_n)=\eta(C_n)=n.$
\end{itemize}
\end{theorem}

\begin{proof}
\begin{itemize}
\item[$(i)$] This result was proved by J.~Olson in \cite{Olso69a} using the notion of group algebra.
The special case of elementary $p$-groups, which says that
$\mathsf{D}(C^r_p)=r(p-1)+1$, can be easily deduced from the
Chevalley-Warning theorem (see \cite{EmdeBoas69} for example).

\medskip
\item[$(ii)$] The value of $\mathsf{D}(.)$ for groups with rank $2$ is also
due to J.~Olson (see \cite{Olso69b}), and uses the special case
$\eta(C^2_p)=3p-2$ with $p$ prime. The complete statement for
$\eta(.)$ has been proved by A.~Geroldinger and F.~Halter-Koch (see \cite{GeroKoch05}, Theorem 5.8.3).
\end{itemize}
\end{proof}

The value of $\eta(.)$ for Abelian $p$-groups with rank $r
\geq 3$ is not known in general, even in the special case of
elementary $p$-groups. It is only known that for every $r \in
\mathbb{N}^*$, we have $\eta(C^r_2)=2^r$, and it is conjectured that
for every odd $p \in \mathcal{P}$, we have $\eta(C^3_p)=8p-7$ and
$\eta(C^4_p)=19p-18$. The interested reader is for instance referred
to \cite{EdelGero06} and \cite{GaoSchmid06}, for a complete account
on this topic.

\medskip
Yet, N.~Alon and M.~Dubiner showed in \cite{Alodub95} an important
theorem related to the constant $\eta(.)$ of elementary $p$-groups.
We will use the following corollary of this result.

\begin{theorem} \label{Alon c_r} For every $r \in \mathbb{N}^*$, there exists a constant
$c_r > 0$ such that for every $p \in \mathcal{P}$, the following
holds:
$$\eta\left(C^r_p\right) \leq c_r(p-1)+1.$$
\end{theorem}

\bigskip In this paper, we will study the \em cross number \em of finite Abelian groups. For this purpose,
we recall some definitions and also the results known so far, to the
best of our knowledge, concerning this constant. Let $G$ be a finite
Abelian group. If $G \simeq C_{\nu_1} \oplus \dots \oplus
C_{\nu_s}$, with $\nu_i > 1$ for all $i \in \llbracket 1,s
\rrbracket$, is the longest possible decomposition of $G$ into a
direct product of cyclic groups, then we set:
$$\mathsf{k}^{*}(G)=\displaystyle\sum_{i=1}^s
\frac{\nu_i-1}{\nu_i},$$ and
$$\mathsf{K}^{*}(G)=\displaystyle\sum_{i=1}^s
\frac{\nu_i-1}{\nu_i}+\frac{1}{\exp(G)}=\mathsf{k}^{*}(G)+\frac{1}{\exp(G)}
.$$

\medskip
\noindent The \em cross number \em of a sequence
$S=(g_1,\dots,g_l)$, denoted by $\mathsf{k}(S)$, is defined by:
$$\mathsf{k}(S)=\displaystyle\sum_{i=1}^{l} \frac{1}{\text{ord}(g_i)}.$$

\medskip
\noindent Then, we define the \em little cross number \em
$\mathsf{k}(G)$ of $G$:

$$\mathsf{k}(G)=\max\{\mathsf{k}(S) | S \text{ zero-sumfree sequence of } G \},$$

\medskip
\noindent as well as the \em cross number \em of $G$, denoted by
$\mathsf{K}(G)$:

$$\mathsf{K}(G)=\displaystyle\max \{\mathsf{k}(S)| S \text{
minimal zero-sum sequence of } G\}.$$

\medskip
The cross number was introduced by U.~Krause in \cite{Krause84} in
order to clarify the relationship between the arithmetic of a Krull
monoid and the properties of its ideal class group. For this reason,
the cross number plays a key r\^{o}le in the theory of non-unique
factorization (see \cite{Krause84}, \cite{GaoGero00},
\cite{GeroHam02}, \cite{Schmid05}, \cite{Plagne05bis}, \cite{Plagne05} and
\cite{Schmid06} for some applications of the cross number, the
surveys \cite{Chap99},\cite{GeroKoch06} and the book
\cite{GeroKoch05} which presents exhaustively the different aspects
of the theory).

\medskip
For the sake of completeness, we mention that the cross number has
been studied in other directions also (see for example
\cite{ChapGero96}, \cite{GeroSchnCross97} and \cite{Bag04}), and that this concept arose in
a natural way in combinatorial number theory (see for instance
\cite{Gero93} and
\cite{Elledge05}).

\medskip
Given a finite Abelian group $G$, a natural construction
(see \cite{Krause84} or \cite{GeroKoch05}, Proposition 5.1.8) gives
the following lower bounds:
$$\mathsf{k}^{*}(G) \leq \mathsf{k}(G) \hspace{0.5cm} \text{ and } \hspace{0.5cm} \mathsf{K}^{*}(G) \leq \mathsf{K}(G),$$
yet, except for Abelian $p$-groups (see \cite{GeroCross94})
and other special cases (see \cite{GeroSchnCross94}), the exact
values of the cross and little cross numbers are also unknown in
general, even for cyclic groups. In addition, still no
counterexample is known for which equality does not hold in the
previous inequalities, which would allow us to disprove the
following conjecture.

\begin{conjecture}\label{Conjecture Krause} For every finite Abelian group $G$, one has the
following: $$\mathsf{k}^{*}(G)=\mathsf{k}(G) \hspace{0.5cm} \text{
and } \hspace{0.5cm} \mathsf{K}^{*}(G)=\mathsf{K}(G).$$
\end{conjecture}

Regarding upper bounds, and since the constants $\mathsf{k}(.)$ and
$\mathsf{K}(.)$ are closely related to each other, it suffices,
according to the following proposition (see \cite{GeroKoch05},
Proposition 5.1.8), to bound from above the little cross number so
as to bound from above the cross number, but also the Davenport
constant. Since $\mathsf{k}(.)$ is easier to handle, one usually
prefers to study the cross number via the little cross number, and
we will do so in this paper.

\begin{prop}\label{majoration de K et Dav par k} Let $G$ be a finite Abelian group with
$\exp(G)=n$. Then, the two following statements hold:

\begin{itemize}
\item[$(i)$] $$\mathsf{k}(G) + \frac{1}{n} \leq \mathsf{K}(G) \leq \mathsf{k}(G) + \frac{1}{P^-(n)},$$
\item[$(ii)$] $$\mathsf{D}(G) \leq n\mathsf{k}(G)+1.$$
\end{itemize}
\end{prop}

\medskip
Two types of upper bounds are currently known for $\mathsf{k}(.)$.
The first one holds for any finite Abelian group, and was obtained
by A.~Geroldinger and R.~Schneider in \cite{GeroSchnCross96} and in
\cite{GeroKoch05}, Theorem 5.5.5, using character theory and the
notion of group algebra.

\begin{theorem}\label{Majorant Geroldinger-Schneider}
Let $G$ be a finite Abelian group with $\exp(G)=n$. Then, for every
$d \in \mathcal{D}_{n}$, one has the following:
$$\mathsf{k}(G) \leq \frac{d-1}{P^-(n)} + \log \left(\frac{|G|}{d}\right).$$
In particular $\mathsf{k}(G) \leq \log|G|$.
\end{theorem}

Eventhough this upper bound is general and easy to compute, it does
not really fit what we know about the behaviour of the cross number.
For example, let $r > 1$ be an integer. If we consider an elementary
$p$-group with rank $r$, it is known that
$\mathsf{k}\left(C^r_p\right)=\mathsf{k}^*\left(C^r_p\right)\leq r,$
yet $\log\left(|C^r_p|/p\right)=(r-1)\log p$ diverges when $p$ tends
to infinity.

\medskip
From this point of view, and in the special case of cyclic groups, a
more precise upper bound was found by U.~Krause and C.~Zahlten in
\cite{Krause91} which, expressed with our notations, gives the
following.

\begin{theorem}
For every $n \in \mathbb{N}^*$, one has the following:
$$\mathsf{k}(C_n) \leq 2\omega(n).$$
\end{theorem}

It should be underlined that this upper bound has the right order of
magnitude, since one has $\mathsf{k}(C_n) \geq \mathsf{k}^*(C_n)
\geq \omega(n)/2$ by definition.

\section{New results and plan of the paper}\label{Section
Introduction}
In this paper, we generalize the work of \cite{Krause91} to every
finite Abelian group so as to obtain a new upper bound for the
little cross number in the general case, which no longer depends on
the cardinality of the group considered, and which supports the
conjecture that the little cross number of a finite Abelian group
$G$ with rank $r$ and exponent $n$ is less than $r\omega(n)$.

\medskip For this purpose, we introduce the two following constants.
Let $G$ be a finite Abelian group and $d',d \in \mathbb{N}^*$ be two
integers such that $d \in \mathcal{D}_{\exp(G)}$ and $d' \in
\mathcal{D}_d$.

\medskip
By $\mathsf{D}_{(d',d)}(G)$ we denote the smallest integer $t \in
\mathbb{N}^*$ such that every sequence $S$ of $G_d$ with length $|S|
\geq t$ contains a subsequence of sum in $G_{d/d'}$.

\medskip
By $\eta_{(d',d)}(G)$ we denote the smallest integer $t \in
\mathbb{N}^*$ such that every sequence $S$ of $G_d$ with length $|S|
\geq t$ contains a subsequence $S' \subseteq S$ of length $|S'| \leq
d'$ and of sum in $G_{d/d'}$.

\medskip
To start with, we will prove in Section \ref{Section D et Eta}
(Proposition \ref{propmarrantegenerale}), that for any finite
Abelian group $G$ and every $1 \leq d' \text{ } | \text{ } d \text{
} | \text{ } \exp(G)$, $\mathsf{D}_{(d',d)}(G)$ and
$\eta_{(d',d)}(G)$ are linked to the constants $\mathsf{D}(.)$ and
$\eta(.)$ of a particular subgroup $G_{\upsilon(d',d)}$ of $G$.

\medskip
In Section \ref{Section Main Theorem}, we will prove the main
theorem (Theorem \ref{Profils}). This result will be stated at the
end of this section. Before giving this general and technical
theorem, we emphasize the many consequences it has.

\medskip
To obtain these results, we introduce the two following arithmetic
functions:
$$\alpha(n)=\displaystyle\sum_{d \in \mathcal{D}_n}
\frac{P^-(d)-1}{d} \text{ } \text{ and } \text{ }
\beta(n)=\displaystyle\sum_{d \in \mathcal{D}_n \cap \mathcal{P}}
\frac{P^-(d)-1}{d},$$

\noindent which will be investigated in Section \ref{Section Alpha
Beta}. In particular, simple upper bounds for these functions lead,
by applying the main theorem, to the following qualitative result,
proved in Section \ref{Section Majoration}.

\begin{prop} \label{Majoration du cross number par omega}
For every $r \in \mathbb{N}^*$ there exists a constant $d_r
> 0$ such that, for every finite Abelian group $G$ with $r(G) \leq r$ and
$\exp(G)=n$, the following holds:
$$\mathsf{k}(G) \leq
d_r\omega(n).$$
\end{prop}

Consequently, when considering the cross number of a finite Abelian
group $G$ with fixed or bounded rank, Proposition \ref{Majoration du
cross number par omega} gives a qualitative upper bound which
depends only on the number of distinct prime divisors $\omega(n)$ of
$\exp(G)=n$, and which improves, at least asymptotically, the one stated in Theorem \ref{Majorant Geroldinger-Schneider}, since the function $\omega$ can have arbitrary small values in
$\mathbb{N}^*$ even for arbitrary large $n$, but mainly since it
is known (see for instance Chapter I.5 of the book \cite{Tenenb})
that one has:
$$\omega(n) \lesssim \frac{\log n}{\log\log n}.$$

\medskip
In addition, more accurate upper bounds for some sequences built
with $\alpha(n)$ and $\beta(n)$, obtained in Lemma \ref{lemme alpha
beta 1erepartie} (see Section \ref{Section Alpha Beta}), enable us
to prove the following quantitative result (see Section \ref{Section
Majoration}) which states that when $r=1$ or $2$, one can choose
$d_r$ in the following way:

$$d_1=\frac{166822111}{109486080} \approx 1.5237 \hspace{0.5cm}
\text{ and } \hspace{0.5cm} d_2=\frac{1784073894563}{476759162880}
\approx 3.7421.$$

\medskip
\noindent Once $d_1$ and $d_2$ are defined in such a way, one can
state the following proposition.
\begin{prop} \label{Majoration du cross number des groupes de rang 2}
\begin{itemize}
\item[$(i)$]
For every cyclic group $G \simeq C_n$, $n \in \mathbb{N}^*$, we
have: $$\mathsf{k}(G) \leq \alpha(n) \leq d_1\omega(n).$$

\item[$(ii)$] For every finite Abelian group $G \simeq C_m
\oplus C_n$, with $1 < m \text{ } | \text{ } n \in \mathbb{N}^*$, we
have: $$\mathsf{k}(G) \leq 3\alpha(n)-\beta(n) \leq d_2\omega(n).$$
\end{itemize}
\end{prop}

Moreover, the asymptotical behaviours of $\alpha(n)$ and $\beta(n)$,
studied in Lemma \ref{lemme alpha beta 2emepartie}, imply several
asymptotical results, some of them being sharp, concerning the cross
and little cross numbers as well as the Davenport constant. In
particular, these results show that Conjecture \ref{Conjecture
Krause} holds asymptotically in at least two different directions.
These results will be proved in Section \ref{Section resultats
asymptotiques}, and in order to state them, we will need the
following notation. For every $r \in \mathbb{N}^*$ and
$l_1,\dots,l_r \in \mathbb{N}^*$, we set:
$$\mathcal{E}_{(l_1,\dots,l_r)}=\left\{\bigoplus^r_{i=1}C_{n_i}, \text{ } 1 < n_1|\dots|n_r \in \mathbb{N} \text{ }|
\text{ } \forall i \in \llbracket 1,r \rrbracket, \text{ }
\omega(n_i)=l_i \text{ and }
\text{gcd}\left(n_i,\frac{n_r}{n_i}\right)=1\right\}.$$

\begin{prop}\label{majoration en fonction des li} For every $r \in \mathbb{N}^*$ and $l_1,\dots,l_r \in \mathbb{N}^*$,
the following statements hold:
\begin{itemize}
\item[$(i)$] $$\displaystyle\lim_{\tiny{\begin{array}{c}
P^-(n_r) \rightarrow +\infty \\
C_{n_1} \oplus \dots \oplus C_{n_r} \in
\mathcal{E}_{(l_1,\dots,l_r)}
\end{array}}} \mathsf{k}(C_{n_1} \oplus \dots \oplus C_{n_r}) = \displaystyle\sum^r_{i=1}l_i,$$

\item[$(ii)$] $$\displaystyle\lim_{\tiny{\begin{array}{c}
P^-(n_r) \rightarrow +\infty \\
C_{n_1} \oplus \dots \oplus C_{n_r} \in
\mathcal{E}_{(l_1,\dots,l_r)}
\end{array}}} \mathsf{K}(C_{n_1} \oplus \dots \oplus C_{n_r}) = \displaystyle\sum^r_{i=1}l_i,$$

\item[$(iii)$]
$$\displaystyle\limsup_{\tiny{\begin{array}{c} P^-(n_r) \rightarrow +\infty \\
C_{n_1} \oplus \dots \oplus C_{n_r} \in
\mathcal{E}_{(l_1,\dots,l_r)} \end{array}}} \frac{\mathsf{D}(C_{n_1}
\oplus \dots \oplus C_{n_r})}{n_r} \leq
\displaystyle\sum^r_{i=1}l_i.$$
\end{itemize}
\end{prop}

Concerning the groups of the form $C^r_n$, we obtain the following
corollary by specifying $n_1=\dots=n_r$ in Proposition
\ref{majoration en fonction des li}.

\begin{prop} \label{majoration C^r_n} For all integers $r,l \in \mathbb{N}^*$ the three following statements hold:
\begin{itemize}
\item[$(i)$] $$\displaystyle\lim_{\tiny{\begin{array}{c} P^-(n) \rightarrow +\infty \\ \omega(n)=l \end{array}}} \mathsf{k}(C^r_n) = rl,$$

\item[$(ii)$] $$\displaystyle\lim_{\tiny{\begin{array}{c} P^-(n) \rightarrow +\infty \\ \omega(n)=l \end{array}}} \mathsf{K}(C^r_n) = rl,$$

\item[$(iii)$] $$\displaystyle\limsup_{\tiny{\begin{array}{c} P^-(n) \rightarrow +\infty \\ \omega(n)=l \end{array}}} \frac{\mathsf{D}(C^r_n)}{n} \leq rl.$$
\end{itemize}
\end{prop}

It may be observed that Proposition \ref{majoration en fonction des
li} and Proposition \ref{majoration C^r_n} are somehow reminiscent
of \cite{GeroSchnCross96}, Theorem 2(b), since this result and our
Proposition \ref{majoration en fonction des li} give the value of
the cross number of "large" groups. However, a more precise look at
both results shows that they are of a different nature. Indeed,
while A.~Geroldinger and R.~Schneider's result is not asymptotical
but valid only for special groups satisfying some restrictive
conditions, ours, although of asymptotical nature, is valid in a
wider framework.

\medskip
The following proposition will also be proved in Section
\ref{Section resultats asymptotiques}.
\begin{prop} \label{majoration
C^r_n bis} For all $r \in \mathbb{N}^*$, the two following
statements hold:
\begin{itemize}
\item[$(i)$] $$\displaystyle\lim_{\tiny{\begin{array}{c} \omega(n) \rightarrow +\infty \end{array}}} \frac{\mathsf{k}(C^r_n)}{\omega(n)} = r,$$

\item[$(ii)$] $$\displaystyle\lim_{\tiny{\begin{array}{c} \omega(n) \rightarrow +\infty \end{array}}} \frac{\mathsf{K}(C^r_n)}{\omega(n)} = r.$$
\end{itemize}
\end{prop}

All these results are deduced from the following proposition, proved
in Section \ref{Section Majoration} under the stronger form of
Proposition \ref{majoration qualitative generale}, and which is a
somewhat rough corollary of the main theorem (Theorem
\ref{Profils}). For the sake of clarity, we recall that the constant
$c_r$ is the one which has been introduced in Theorem \ref{Alon
c_r}.

\begin{prop} \label{majoration qualitative} Let $G$
be a finite Abelian group with $r(G)=r$ and $\exp(G)=n$. We set
$H=C^r_n$ and also:
$$\varphi(G,H)=\begin{cases}
\mathsf{k}^*(H/G)     \hspace{0.48cm} \text{ if } G \text{ is a direct summand of } H,\\
\mathsf{k}^*(H/G)/n      \text{ otherwise}.
  \end{cases}$$

\noindent Then, one has the following upper bound for the little
cross number $\mathsf{k}(G)$:
$$\mathsf{k}(G) \leq c_r\big(\alpha(n)-\beta(n)\big) + r\beta(n) - \varphi(G,H).$$
\end{prop}

\medskip
The main theorem of this paper (Theorem \ref{Profils}) will be
proved in Section \ref{Section Main Theorem}. In order to state it,
we will need the following definitions and notations which will be
extensively used in Sections \ref{Section D et Eta} and \ref{Section
Main Theorem}.

\medskip
Let $G \simeq C_{n_1} \oplus \dots \oplus C_{n_r},$ with $1 < n_1
\text{ } |\text{ } \dots \text{ }|\text{ } n_r \in \mathbb{N}$, be a
finite Abelian group with $\exp(G)=n$, $\tau(n)=m$ and $d',d \in
\mathbb{N}^*$ be such that $d \in \mathcal{D}_{n}$ and $d' \in
\mathcal{D}_d$. For all $i \in \llbracket 1,r \rrbracket$, we set:
$$A_i=\gcd(d',n_i), \text{ }
B_i=\frac{\mathrm{lcm}(d,n_i)}{\mathrm{lcm}(d',n_i)},$$

$$\upsilon_i(d',d)=\frac{A_i}{\gcd(A_i,B_i)},$$

\noindent and $$G_{\upsilon(d',d)}=C_{\upsilon_1(d',d)} \oplus \dots
\oplus C_{\upsilon_r(d',d)}.$$

\medskip
\noindent Then, for every $d \in
\mathcal{D}_{n}=\left\{d_1,\dots,d_{m}\right\}$ and
$x=\left(x_{d_1},\dots,x_{d_{m}}\right) \in \mathbb{N}^{m}$, we set:

$$f_d(x) = \displaystyle\min_{d' \in \mathcal{D}_d\backslash\{1\}}
\left(\eta(G_{\upsilon(d',d)})\right)-1 - x_d,$$

\medskip
$$g_d(x) = \mathsf{D}(G_{\upsilon(d,d)})-1 - \sum_{d' \in \mathcal{D}_d}
x_{d'},$$
and
$$h(x)=\sum_{d \in \mathcal{D}_n}
\frac{x_{d}}{d} - \mathsf{k}^*(G).$$

\medskip
We can now state the main theorem.
\begin{theorem}\label{Profils}
Let $G \simeq C_{n_1} \oplus \dots \oplus C_{n_r},$ with $1 < n_1
\text{ } |\text{ } \dots \text{ }|\text{ } n_r \in \mathbb{N}$, be a
finite Abelian group with $\exp(G)=n$ and $\tau(n)=m$. For every
zero-sumfree sequence $S$ of $G$ reaching the maximum
$\mathsf{k}(S)=\mathsf{k}(G)$, and being of minimal length
regarding this property, the $m$-tuple
$x=\left(|S_{d_1}|,\dots,|S_{d_m}|\right)$ is an element of the
polytope $\mathbb{P}_G \cap \mathbb{H}_G$ where:

$$\mathbb{P}_G=\{x \in \mathbb{N}^m \text{ } |
\text{ } f_d(x) \geq 0, \text{ } g_d(x) \geq 0, \text{ } d \in
\mathcal{D}_n\},$$ and $$\mathbb{H}_G=\{x \in \mathbb{N}^m \text{ }
| \text{ } h(x) \geq 0\}.$$
\end{theorem}

Keeping the notations of Theorem \ref{Profils}, we obtain the
following immediate corollary, which gives a general upper bound for
the little cross number of a finite Abelian group, expressed as the
solution of an integer linear program.

\begin{corollary} \label{Majoration du cross number} For every finite Abelian group $G$, one has the following upper bound:
$$\mathsf{k}(G) \leq \displaystyle\max_{x \in \mathbb{P}_G} \left(\displaystyle\sum^m_{i=1}
\frac{x_{d_i}}{d_i}\right).$$
\end{corollary}

In principle, the wide generality of Theorem \ref{Profils} and
Corollary \ref{Majoration du cross number} leaves a good hope that
it could lead to new - and maybe optimal - upper bounds for
$\mathsf{k}(G)$ in the general case. However, such improvements will
require a precise study of the polytope $\mathbb{P}_G$, which is
certainly a complicated, but not hopeless, task.

\section{On the quantities $\mathsf{D}_{(d',d)}(G)$ and
$\eta_{(d',d)}(G)$}\label{Section D et Eta}

In this section, we will denote by $\pi_i$, for all $i \in
\llbracket 1,r \rrbracket$, the canonical epimorphism from $C_{n_i}$
to $C_{\upsilon_i(d',d)}$. Although this epimorphism clearly depends
on $d'$ and $d$, we do not emphasize this dependence here since
there is no risk of ambiguity.
Moreover, one can notice that whenever $d$ divides $n_i$, we have
$\upsilon_i(d',d)=\gcd(d',n_i)=d'$, and in particular
$\upsilon_r(d',d)=d'.$ In the sequel, when $d'=d$, we will write
$\upsilon_i(d)$ instead of $\upsilon_i(d,d).$

\begin{lem} \label{lemme arithmetique modulaire} Let $G \simeq C_{n_1} \oplus \dots \oplus C_{n_r}$, with $1 < n_1
\text{ } |\text{ } \dots \text{ }|\text{ } n_r \in \mathbb{N}$, be a
finite Abelian group and $d',d \in \mathbb{N}^*$ be such that $d \in
\mathcal{D}_{\exp(G)}$ and $d' \in \mathcal{D}_d.$
Then, for every $g=[a_1,\dots,a_r] \in G$, we have:
$$\frac{d}{d'} \left[\frac{n_1}{\gcd(d,n_1)} a_1,\dots,\frac{n_r}{\gcd(d,n_r)} a_r\right]=0
\text{ if and only if } \pi_i(a_i)=0 \text{ for all } i \in
\llbracket 1,r \rrbracket .$$
\end{lem}

\begin{proof}
First, we have the following equalities:
\begin{eqnarray*}
\frac{d}{d'}\frac{n_i}{\gcd(d,n_i)} & = &
\frac{\mathrm{lcm}(d,n_i)}{d'}\\
&=& \frac{\mathrm{lcm}(d,n_i)n_i}{d' n_i}\\
&=& \frac{\mathrm{lcm}(d,n_i)n_i}{\gcd(d',n_i)\mathrm{lcm}(d',n_i)}\\\\
&=& B_i\frac{n_i}{A_i} \in \mathbb{N}.
\end{eqnarray*}

\medskip
Let $[a_1,\dots,a_r] \in G$ be such that:

$$\frac{d}{d'} \left[\frac{n_1}{\gcd(d,n_1)} a_1,\dots,\frac{n_r}{\gcd(d,n_r)} a_r\right]=0.$$

\medskip
\noindent For all $i \in \llbracket 1,r \rrbracket$, one has:
$$\frac{d}{d'} \frac{n_i}{\gcd(d,n_i)} a_i = B_i\frac{n_i}{A_i}a_i =
0,$$

\noindent which is equivalent, considering $a_i$ as an integer, to
the following relation:
$$A_i | B_i a_i,$$

\noindent that is to say, dividing each side by $\gcd(A_i,B_i)$,
that one has:
$$\upsilon_i(d',d) \text{ } \Big{|} \text{ } \frac{B_i}{\gcd(A_i,B_i)} a_i,$$

\noindent which, since:
$$\gcd\left(\frac{A_i}{\gcd(A_i,B_i)},\frac{B_i}{\gcd(A_i,B_i)}\right)=1,$$

\noindent is equivalent to: $$\upsilon_i(d',d) | a_i,$$

\noindent and the desired result is proved.
\end{proof}

\begin{prop} \label{propmarrantegenerale} Let $G \simeq C_{n_1} \oplus \dots \oplus C_{n_r}$, with $1 < n_1
\text{ } |\text{ } \dots \text{ }|\text{ } n_r \in \mathbb{N}$, be a
finite Abelian group and $d',d \in \mathbb{N}^*$ be such that $d \in
\mathcal{D}_{\exp(G)}$ and $d' \in \mathcal{D}_d$. Then, we have the
two following equalities:
$$\begin{cases}\mathsf{D}_{(d',d)}(G)=\mathsf{D}\left(C_{\upsilon_1(d',d)} \oplus
\dots \oplus C_{\upsilon_r(d',d)}\right),\\
\mathsf{\eta}_{(d',d)}(G)=\mathsf{\eta}\left(C_{\upsilon_1(d',d)}
\oplus \dots \oplus C_{\upsilon_r(d',d)}\right).\end{cases}$$
\end{prop}

\begin{proof} Let $[a_1,\dots,a_r] \in G_d$. We know that $\text{ord}\left([a_1,\dots,a_r]\right)=\mathrm{lcm}(\text{ord}(a_1),\dots,\text{ord}(a_r))$,
and so $\text{ord}\left([a_1,\dots,a_r]\right) | d$ implies
$\text{ord}(a_i) | d$ for all $i \in \llbracket 1,r \rrbracket$.

By Lagrange theorem, we also have $\text{ord}(a_i) | n_i$, which
implies that: $$\text{ord}(a_i) | \gcd(d,n_i) \text{ for all } i
\in \llbracket 1,r \rrbracket,$$ and since any cyclic group
$C_{n_i}$ contains a unique subgroup of order $\gcd(d,n_i)$, we can
write:
$$a_i=\frac{n_i}{\gcd(d,n_i)} a'_i \text{ with }
a'_i \in C_{n_i}.$$

\noindent We now consider a sequence $S=(g_1,\dots,g_m)$ of $G_d$
with $m \in \mathbb{N}^*$. According to the previous argument, the
elements of $S$ have the following form:

$$\begin{array}{ccccc}
g_1 & = & [a_{1,1},\dots,a_{1,r}] & = &
\left[\frac{n_1}{\gcd(d,n_1)}
a'_{1,1},\dots,\frac{n_r}{\gcd(d,n_r)} a'_{1,r}\right],\\
\vdots & & \vdots & & \vdots \\
g_m & = & [a_{m,1},\dots,a_{m,r}] & = &
\left[\frac{n_1}{\gcd(d,n_1)} a'_{m,1},\dots,\frac{n_r}{\gcd(d,n_r)}
a'_{m,r}\right].
\end{array}$$

\noindent Let $K$ be a nonempty subset of $\{1,\dots,m\}$. Then, the
sum $\sum_{k \in K} \left[a_{k,1},\dots,a_{k,r}\right]$ is an
element of $G_{d/d'}$ if and only if:

$$\frac{d}{d'} \displaystyle\sum_{k \in K}
\left[a_{k,1},\dots,a_{k,r}\right] =\frac{d}{d'}
\left[\frac{n_1}{\gcd(d,n_1)} \displaystyle\sum_{k \in K}
a'_{k,1},\dots,\frac{n_r}{\gcd(d,n_r)} \displaystyle\sum_{k \in K}
a'_{k,r}\right]=0,$$

\noindent and by Lemma \ref{lemme arithmetique modulaire}, this
relation is equivalent to:

$$\displaystyle\sum_{k \in K}
\left[\pi_1(a'_{k,1}),\dots,\pi_r(a'_{k,r})\right]=0 \text{ in }
\displaystyle\bigoplus^r_{i=1} C_{\upsilon_i(d',d)}.$$

\noindent Therefore, from the definition of the constant
$\mathsf{D}(.)$, one can deduce that the smallest integer $m \in
\mathbb{N}^*$ such that for every sequence $S=(g_1,\dots,g_m)$ of
$G_d$ with length $|S| \geq m$, there exists a nonempty subset $K
\subseteq \{1,\dots,m\}$ such that $\sum_{k \in K}
\left[a_{k,1},\dots,a_{k,r}\right]$ is an element of $G_{d/d'}$ is
exactly $\mathsf{D}\left(C_{\upsilon_1(d',d)} \oplus \dots \oplus
C_{\upsilon_r(d',d)}\right)$. This proves the first equality.

If moreover, one expects the additional condition $|K| \leq
\upsilon_r(d',d)=d'$ to be verified, then, by the definition of
$\mathsf{\eta}(.)$, the corresponding smallest possible integer $m
\in \mathbb{N}^*$ is exactly
$\mathsf{\eta}\left(C_{\upsilon_1(d',d)} \oplus \dots \oplus
C_{\upsilon_r(d',d)}\right)$, which proves the second equality.
\end{proof}

\section{Proof of the main theorem}\label{Section Main Theorem}
\begin{proof}[Proof of Theorem \ref{Profils}] Let $S$ be a zero-sumfree sequence of $G$ verifying $\mathsf{k}(S)=\mathsf{k}(G),$
and being of minimal length regarding this property. For every
$d \in \mathcal{D}_n$, we set $x_d=|S_d|$, and we suppose that the
$m$-tuple $x=(x_{d_1},\dots,x_{d_m})$ is not an element of the
polytope $\mathbb{P}_G \cap \mathbb{H}_G$. Thus, one has the three
following cases.

\medskip $\textbf{Case 1.}$ There exists $d_0 \in
\mathcal{D}_{n}$ such that $f_{d_0}(x)<0$. Therefore, it exists
$d'_0 \in \mathcal{D}_{d_0}\backslash\{1\}$ verifying $x_{d_0} \geq
\eta\left(C_{\upsilon_1(d'_0,d_0)} \oplus \dots \oplus
C_{\upsilon_r(d'_0,d_0)}\right)$ which means, by Proposition
\ref{propmarrantegenerale}, that $x_{d_0} \geq
\eta_{(d'_0,d_0)}(G)$. So, the sequence $S$ contains $X$ elements of
order $d_0$, with $1 < X \leq d'_0$, the sum of which is an element
of order $\tilde{d_0}$ dividing $d_0/d'_0$.

\medskip
Let $S'$ be the sequence obtained from $S$ by replacing these $X$
elements by their sum. In particular, we have $|S'|=|S|-X+1<|S|$.
Moreover, $S'$ is a zero-sumfree sequence and verifies the following
equalities:

\begin{center} $\begin{cases}
|S'_{d_0}|=|S_{d_0}|-X, \\
|S'_{\tilde{d_0}}|=|S_{\tilde{d_0}}|+1, \\
|S'_{d}|=|S_{d}| \text{ } \forall d \neq d_0,\tilde{d_0}.
\end{cases}$
\end{center}

\noindent Since $$\frac{1}{\tilde{d_0}}-\frac{X}{d_0} \geq 0,$$

\noindent one has the following inequalities:

\begin{eqnarray*}\mathsf{k}(S) & = & \displaystyle\sum_{d \in
\mathcal{D}_{\exp(G)}}
\frac{x_d}{d} \\
& \leq & \displaystyle\sum_{d \in \mathcal{D}_{\exp(G)} \backslash
\{d_0,\tilde{d_0}\}} \frac{x_d}{d} +
\frac{x_{\tilde{d_0}}+1}{\tilde{d_0}} +
\frac{x_{d_0}-X}{d_0}\\\\
& = & \mathsf{k}(S').\end{eqnarray*}

\noindent So, we obtain $\mathsf{k}(S')=\mathsf{k}(G)$ and
$|S'|<|S|$, which is a contradiction.

\medskip $\textbf{Case 2.}$ There exists $d_0 \in \mathcal{D}_{n}$ such that $g_{d_0}(x) < 0$. As a consequence, we have
$\sum_{d \in \mathcal{D}_{d_0}}
x_d \geq \mathsf{D}(C_{\upsilon_1(d_0)} \oplus \dots \oplus
C_{\upsilon_r(d_0)})$ and Proposition \ref{propmarrantegenerale}
gives the existence of a zero-sum subsequence, which is a
contradiction.

\medskip $\textbf{Case 3.}$ One has $h(x) < 0$, that is to say
$\mathsf{k}(S)=\mathsf{k}(G)<\mathsf{k}^*(G)$ which is a
contradiction.
\end{proof}

An interesting special case is the one of finite Abelian groups with
rank $2$. Indeed, for such groups, all the parameters used to define
the polytope $\mathbb{P}_G$ in the main theorem are known by Theorem
\ref{Proprietes D et Eta}:
$$\mathsf{D}_{(d,d)}(G)=\upsilon_1(d)+\upsilon_2(d)-1
\hspace{0.23cm} \text{ and } \hspace{0.31cm}
\eta_{(d',d)}(G)=2\upsilon_1(d',d)+\upsilon_2(d',d)-2,$$

\noindent and therefore allow us to compute an explicit upper bound
for the little cross number $\mathsf{k}(G)$ by linear programming
methods (see for instance the book \cite{Schri98} for an exhaustive
presentation of these methods).

\section{Some sequences related to the exponent of a finite Abelian group}\label{Section Alpha
Beta}
Let $(\alpha_l)_{l \geq 1}$ and $(\beta_l)_{l \geq 1}$ be the two
following sequences of integers, built from the set of prime numbers:
$$\alpha_1=1 \text{ and } \alpha_{l}=1 + \frac{p_l}{p_l-1}\alpha_{l-1} \text{ for all } l \geq 2,$$

\noindent as well as
$$\beta_{l}=\displaystyle\sum^l_{i=1} \frac{p_i-1}{p_i} \text{ for all } l \geq 1.$$

\noindent Finally, we define a third sequence $(\gamma_l)_{l \geq
1}$ in the following fashion:
$$\gamma_l=3\alpha_l-\beta_l \text{ for all } l \geq 1.$$

The first values of $(\alpha_l)_{l \geq 1}$ are the following:
$$\alpha_1=1, \alpha_2=2.5, \alpha_3=4.125, \alpha_4=5.8125,
\alpha_5=7.39375 \text{ etc.}$$
\noindent Since $2l-1 \leq p_l$, we can already show, by induction
on $l$, the following statement:
$$\alpha_l \leq 2l, \text{ for all } l
\geq 1.$$

\noindent Indeed, one has $\alpha_1=1 \leq 2$, and if the statement
is true for $l-1$, we obtain:
$$\alpha_l = 1 + \alpha_{l-1} + \frac{\alpha_{l-1}}{p_l-1} \leq 1 + 2(l-1) + \frac{2(l-1)}{p_l-1} \leq 2l.$$

\medskip
In order to study more precisely the behaviours of $\alpha(n)$ and
$\beta(n)$, we will extensively use a classical lower bound for the
$l$-th prime number, proved by Rosser in \cite{Ros39}, and which is
the following:
$$l \log l \leq p_l \text{ for all } l \geq
1.$$

\medskip
We can now prove Lemma \ref{lemme alpha beta 1erepartie}, which
gives accurate upper bounds for the sequences $(\alpha_l)_{l \geq
1}$ and $(\gamma_l)_{l \geq 1}$, and Lemma \ref{lemme alpha beta
2emepartie}, which states on the one hand that $\alpha_l$ and
$\beta_l$ are both equivalent to $l$ when $l$ tends to infinity, and
on the other hand that when $\omega(n)=l$ is fixed, then both
$\alpha(n)$ and $\beta(n)$ tends to $l$ when $P^-(n)$ tends to
infinity.

\begin{lem} \label{lemme alpha beta 1erepartie} The following statements
hold:
\begin{itemize}
\item[$(i)$] For every integer $n \in \mathbb{N}^*$, with $\omega(n)=l$, we have:
 $$\beta_l \leq \beta(n) \leq \alpha(n) \leq \alpha_l.$$

\item[$(ii)$] For every integer $l \geq 1$, we have:
$$l \leq \alpha_l \leq \frac{\alpha_9}{9}l, \text{ where } \frac{\alpha_9}{9}=\frac{166822111}{109486080}
\approx 1.5237.$$

\item[$(iii)$] For every integer $l \geq 1$, we have:
$$\frac{5}{2}l \leq \gamma_l \leq \frac{\gamma_8}{8}l, \text{ where }
\frac{\gamma_8}{8}=\frac{1784073894563}{476759162880} \approx
3.7421.$$
\end{itemize}
\end{lem}

\begin{proof}
\begin{itemize}
\item[$(i)$] Let $n=q^{m_1}_1 \dots q^{m_l}_l$ be an integer with $q_1 < \dots <
q_l$. Since for all $i \in \llbracket 1,l \rrbracket$, one has $p_i
\leq q_i$, we obtain the first inequality:
$$\beta_l = l-\displaystyle\sum^l_{i=1}\frac{1}{p_i} \leq l-\displaystyle\sum^l_{i=1}\frac{1}{q_i}=\beta(n).$$

The second inequality follows directly from:
$$\beta(n)=\displaystyle\sum_{d \in \mathcal{D}_n \cap \mathcal{P}} \frac{P^-(d)-1}{d}
\leq \displaystyle\sum_{d \in \mathcal{D}_n}
\frac{P^-(d)-1}{d}=\alpha(n).$$

We prove the third inequality by induction on the number of distinct
prime divisors $\omega(n)=l$ of $n$. For $l=1$, the integer $n$ is
of the form $q^{m_1}_1$ and we obtain:
$$\alpha(q^{m_1}_1)=\displaystyle\sum^{m_1}_{i=1}
\frac{q_1-1}{q^{i}_1}=\frac{q^{m_1}_1-1}{q^{m_1}_1} \leq 1=
\alpha_1.$$

\noindent Assume now that the statement is valid for $l-1$.
Therefore, we have: \begin{eqnarray*} \alpha(q^{m_1}_1 \dots
q^{m_l}_l) & = & \frac{q^{m_l}_l-1}{q^{m_l}_l} +
\left(\displaystyle\sum^{m_l}_{i=0} \frac{1}{q^i_l}\right)
\alpha(q^{m_1}_1 \dots
q^{m_{l-1}}_{l-1})  \\
& \leq & \frac{q^{m_l}_l-1}{q^{m_l}_l} + \left(\displaystyle\sum^{m_l}_{i=0} \frac{1}{q^i_l}\right)\alpha_{l-1} \\
& \leq & 1 + \left(\displaystyle\sum^{+\infty}_{i=0} \frac{1}{p^i_l}\right) \alpha_{l-1} \\
& = & 1 + \frac{p_l}{p_l-1}\alpha_{l-1} =\alpha_l,\end{eqnarray*}

\noindent which proves the result.

\medskip
\item[$(ii)$] To start with, it is straightforward that the first inequality $l \leq \alpha_l$ always holds.

\noindent Concerning the second inequality, one has the following:
$$\alpha_{l+1}-\alpha_l  = 1 + \frac{\alpha_l}{p_{l+1}-1} \text{ for all } l \geq
1,$$

\noindent from which we deduce the two following relations:
$$\alpha_l = \alpha_1 + \displaystyle\sum^{l-1}_{k=1} (\alpha_{k+1}-\alpha_{k})= l + \displaystyle\sum^{l-1}_{k=1} \frac{\alpha_k}{p_{k+1}-1},$$
\vspace{-0.8cm}

\noindent as well as
\begin{eqnarray*}
\frac{\alpha_{l+1}}{l+1}-\frac{\alpha_l}{l} & = &
\frac{1}{l+1}+\alpha_l\left(\frac{1}{l+1}\left(1+\frac{1}{p_{l+1}-1}\right)-\frac{1}{l}\right).
\end{eqnarray*}

\noindent In the remainder of this proof, we will set
$\varepsilon(l)=\alpha_l-l=\displaystyle\sum^{l-1}_{k=1}
\frac{\alpha_k}{p_{k+1}-1}$, for all $l \geq 1$.

\noindent Using this notation, we obtain the following:
\begin{eqnarray*}
\frac{\alpha_l}{l}-\frac{\alpha_9}{9} & = &
\displaystyle\sum^{l-1}_{k=9}
\frac{1}{k+1}+\displaystyle\sum^{l-1}_{k=9}\alpha_k
\left(\frac{1}{k+1}\left(1+\frac{1}{p_{k+1}-1}\right)-\frac{1}{k}\right)\\
& = & \displaystyle\sum^{l-1}_{k=9}
\frac{1}{k+1}-\displaystyle\sum^{l-1}_{k=9}
\frac{k+\varepsilon(k)}{k(k+1)}+\displaystyle\sum^{l-1}_{k=9}\frac{\alpha_k}{(p_{k+1}-1)(k+1)}\\
& = & \displaystyle\sum^{l-1}_{k=9} \frac{1}{k+1}\left(\frac{\alpha_k}{(p_{k+1}-1)}-\frac{\varepsilon(k)}{k}\right)\\
& = & \displaystyle\sum^{l-1}_{k=9}
\left(\frac{\varepsilon(k+1)}{k+1}-\frac{\varepsilon(k)}{k}\right)\\
& = & \frac{\varepsilon(l)}{l}-\frac{\varepsilon(9)}{9}.
\end{eqnarray*}

\noindent Moreover, using Rosser's lower bound, we obtain for all $l
\geq 2$:
$$\varepsilon(l) = \displaystyle\sum^{l-1}_{k=1}
\frac{\alpha_k}{p_{k+1}-1}
 \leq \displaystyle\sum^{l-1}_{k=1} \frac{2k}{(k+1)\log (k+1)-1}
\leq 7 + \int^{l}_2 \frac{2\mathrm{d}t}{\log t}=lf(l),$$

\noindent where we set for all $x \in \mathbb{R}$, $x \geq 2$:
$$f(x)=\frac{1}{x}\left(7+\int^x_2 \frac{2\mathrm{d}t}{\log t}\right).$$

\noindent It is readily seen that this function is non-increasing.
Moreover, since:
$$\frac{\varepsilon(9)}{9}=\left(\frac{\alpha_9-9}{9}\right) > \frac{1}{2},$$

\noindent and since $f(l) \leq 1/2$ for all $l \geq 241,$ we obtain:
$$\frac{\alpha_l}{l} \leq \frac{\alpha_9}{9}, \text{ for all } l \geq 241.$$

\noindent On the other hand, an easy computation allows us to verify
that $\alpha_9/9$ is also the maximum value taken by
$(\alpha_l/l)_{l \geq 1}$ on $1 \leq l \leq 240$, which proves the
desired result.

\medskip
\item[$(iii)$] The fact that the first inequality $5l/2 \leq \gamma_l$ always holds is
straightforward.

\noindent Moreover, for all $l \geq 1$, one has the following
equality:
\begin{eqnarray*}
\gamma_{l+1} & = & 3\alpha_{l+1}-\beta_{l+1}\\
& = & 3 + 3\alpha_l
+ \frac{3\alpha_l}{p_{l+1}-1}-\beta_l-1+\frac{1}{p_{l+1}}\\
& = & 2 + \gamma_l + \frac{3\alpha_l}{p_{l+1}-1}+\frac{1}{p_{l+1}}.
\end{eqnarray*}

\noindent Using the inequalities $5l/2 \leq \gamma_l$ and $\alpha_l
\leq \alpha_9l/9$, one can deduce that:
\begin{eqnarray*}
\frac{\gamma_{l+1}}{l+1}-\frac{\gamma_l}{l} & = & \frac{2}{l+1} + \gamma_l\left(\frac{1}{l+1}-\frac{1}{l}\right) + \frac{3\alpha_l}{(l+1)(p_{l+1}-1)}+\frac{1}{p_{l+1}(l+1)}\\
& \leq &
\frac{1}{p_{l+1}(l+1)}\left(-\frac{p_{l+1}}{2}+\frac{\alpha_9l}{3}\left(1+\frac{1}{p_{l+1}-1}\right)+1\right).
\end{eqnarray*}

\noindent We set, for all $x \in \mathbb{R}$, $x \geq 1$:
$$g(x)=-\frac{(x+1)\log(x+1)}{2}+\frac{\alpha_9x}{3}\left(1+\frac{1}{(x+1)\log(x+1)-1}\right)+1.$$

\noindent It is easily seen that this function is non-increasing.
Moreover, since a study of $g$ shows that $g(l) \leq 0$ for all $l
\geq 9333,$ we obtain:
$$\frac{\gamma_{l+1}}{l+1}-\frac{\gamma_l}{l} \leq 0, \text{ for all } l \geq 9333.$$

\noindent On the other hand, an easy computation allows us to verify
that $(\gamma_l)_{l \geq 1}$ is increasing from $l=1$ to $l=8$ and
decreasing from $l=8$ to $l=9333$, which proves the desired result.
\end{itemize}
\end{proof}

\begin{lem} \label{lemme alpha beta 2emepartie} The two following statements
hold:
\begin{itemize}
\item[$(i)$] $$\displaystyle\lim_{l \rightarrow
+\infty}\frac{\alpha_l}{l}=1 \hspace{0.5cm} \text{ and }
\hspace{0.5cm} \displaystyle\lim_{l \rightarrow
+\infty}\frac{\beta_l}{l}=1,$$

\item[$(ii)$] $$\displaystyle\lim_{\tiny{\begin{array}{c} P^-(n)  \rightarrow
+\infty\\\omega(n)=l\end{array}}}\alpha(n)=l \hspace{0.5cm} \text{
and } \hspace{0.5cm} \displaystyle\lim_{\tiny{\begin{array}{c}
P^-(n) \rightarrow +\infty
\\ \omega(n)=l
\end{array}}}\beta(n)=l.$$
\end{itemize}
\end{lem}

\begin{proof}
\begin{itemize}
\medskip
\item[$(i)$] Firstly, for all $l \geq 1$, one has the following inequality:
$$l \leq \alpha_l \leq l + \displaystyle\sum^{l-1}_{k=1} \frac{2k}{p_{k+1}-1},$$

\noindent and since the prime number theorem reads as $p_k \sim
k\log k$, we can deduce that:
$$\displaystyle\sum^{l-1}_{k=1} \frac{k}{p_{k+1}-1} \sim
\displaystyle\sum^{l}_{k=2} \frac{1}{\log k} \sim \frac{l}{\log
l}.$$

\noindent Therefore, when $l$ tends to infinity, we obtain
$\displaystyle\lim_{l \rightarrow +\infty}(\alpha_l/l)=1.$

\bigskip Secondly, we can deduce from Rosser's lower bound that for
every $l \geq 3$, one has:
\begin{eqnarray*} \beta_l & \geq & l-\frac{5}{6}-\displaystyle\sum^l_{i=3}\frac{1}{i \log
i}\\
& \geq & l-2-\log\log l.
\end{eqnarray*}

\noindent Since, on the other hand, one always has $\beta_l \leq l$,
we obtain
$\displaystyle\lim_{l \rightarrow +\infty}(\beta_l/l)=1,$ which is

\vspace{-0.1cm} \noindent the desired result.

\medskip
\item[$(ii)$] The result follows from the very definition of
$\alpha(n)$ and $\beta(n)$.
\end{itemize}
\end{proof}

\section{Upper bounds  for the little cross number}\label{Section Majoration}
As previously stated, the upper bound implied by Theorem
\ref{Profils}, and given in Corollary \ref{Majoration du cross
number}, is expressed as the solution of an integer linear program.
Even if this formulation is more precise than any explicit formula
derived from Theorem \ref{Profils}, one may still like to obtain
such a formula in order to interprete the behaviour of the cross
number. In the present section, we obtain such a formula in
Proposition \ref{majoration qualitative generale}. For the proof of
this result, we will use the following lemma, which can be found in
\cite{GeroKoch05}, Proposition $5.1.11$.

\begin{lem} \label{little cross number d'un sous-groupe} Let $H$ be a finite Abelian group and $G \subseteq H$ a
subgroup. Then, one has:
\begin{itemize}
\item[$(i)$] $$\mathsf{k}(G) + \frac{\mathsf{k}(H/G)}{\exp(G)} \leq
\mathsf{k}(H).$$

\item[$(ii)$] If $G$ is a direct summand of $H$, then:
$$\mathsf{k}(G) + \mathsf{k}(H/G) \leq
\mathsf{k}(H).$$
\end{itemize}
\end{lem}

\medskip
We are now ready to prove the following proposition.
\begin{prop} \label{majoration qualitative generale} Let $G$
be a finite Abelian group with $r(G)=r$ and $\exp(G)=n$. We set
$H=C^r_n$ and also:
$$\varphi(G,H)=\begin{cases}
\mathsf{k}^*(H/G)     \hspace{0.48cm} \text{ if } G \text{ is a direct summand of } H,\\
\mathsf{k}^*(H/G)/n      \text{ otherwise}.
  \end{cases}$$

\noindent Then, one has the following upper bound for the little
cross number $\mathsf{k}(G)$:
$$\mathsf{k}(G) \leq \displaystyle\sum_{d \in \mathcal{D}_{n}}
\frac{\min\left(\eta\big(C^r_{P^-(d)}\big),\mathsf{D}(C^r_d)\right)-1}{d}
- \varphi(G,H).$$
\end{prop}

\begin{proof} Since the group $G$ can be injected
in the group $H=C^r_n$, one obtains, applying Lemma \ref{little
cross number d'un sous-groupe}, the relation $\mathsf{k}(G) +
\varphi(G,H) \leq \mathsf{k}(H)$. Then, the desired result follows
from Theorem \ref{Profils} applied to $H$.
\end{proof}

One can notice that for all $r \in \mathbb{N}^*$ and every $p \in
\mathcal{P}$, one always has $\mathsf{D}(C^r_p)\leq
\eta\big(C^r_p\big)$, by definition. Therefore, if we consider an
elementary $p$-group with rank $r$, we obtain:
$$\mathsf{k}\left(C^r_p\right) \leq \displaystyle\sum_{d \in
\mathcal{D}_p} \frac{\mathsf{D}(C^r_d)-1}{d}=\frac{r(p-1)}{p}=
\mathsf{k}^*\left(C^r_p\right).$$ Let $G$ be a finite Abelian group
with $r(G)=r$ and $\exp(G)=n$. Using Theorem \ref{Proprietes D et
Eta}, one obtains that if $r=1$, then for all $d \in \mathcal{D}_n
\backslash \mathcal{P}$, we have $\mathsf{D}(C_d)\geq
\eta\big(C_{P^-(d)}\big)$. Moreover, when $r=2$, then for all $d \in
\mathcal{D}_n \backslash \mathcal{P}$, one has the following
inequality:
$$\mathsf{D}(C^2_d)=2d-1 \geq 3P^-(d)-2 =\eta\big(C^2_{P^-(d)}\big).$$

\noindent Yet, as soon as $r \geq 3$, and except for special types
groups, it becomes more complicated to know exactly, for a given $d$
in $\mathcal{D}_n \backslash \mathcal{P}$, what is the minimum of
$\eta\big(C^r_{P^-(d)}\big)$ and $\mathsf{D}(C^r_d)$.
For this reason, Theorem \ref{Profils} and Proposition
\ref{majoration qualitative generale} remain, in general, really
stronger than Proposition \ref{majoration qualitative}, which we are
going to prove now. Even so, we will see in the next section that
Proposition \ref{majoration qualitative} implies sharp asymptotical
results on the little cross number and the cross number.

\begin{proof}[Proof of Proposition \ref{majoration qualitative}]
Applying Proposition \ref{majoration qualitative generale} and
Theorem \ref{Alon c_r}, we obtain the desired result in the
following manner:
\begin{eqnarray*}\mathsf{k}(G)+\varphi(G,H) & \leq &
\displaystyle\sum_{d \in \mathcal{D}_{n}}
\frac{\min\left(\eta\big(C^r_{P^-(d)}\big),\mathsf{D}(C^r_d)\right)-1}{d}\\
& \leq & \displaystyle\sum_{d \in \mathcal{D}_{n}\cap \mathcal{P}}
\frac{r(P^-(d)-1)}{d} + \displaystyle\sum_{d \in \mathcal{D}_{n}
\backslash \mathcal{P}}
\frac{c_r(P^-(d)-1)}{d}\\\\
& = & c_r\big(\alpha(n)-\beta(n)\big)+r\beta(n).\end{eqnarray*}
\end{proof}

We can now prove the announced qualitative upper bound.
\begin{proof}[Proof of Proposition \ref{Majoration du cross number par omega}]
Since, by the definitions of Section \ref{Section Alpha Beta}, one
always has the following straightforward inequalities:
$$\frac{\omega(n)}{2} \leq \beta(n) \hspace{0.5cm} \text{ and }
\hspace{0.5cm} \alpha(n) \leq 2\omega(n),$$

\noindent we can deduce, by Proposition \ref{majoration qualitative}
and the inequality $r \leq c_r$, the following relation:
$$\mathsf{k}(G) \leq c_r \big(\alpha(n)-\beta(n)\big)+r\beta(n) \leq \left(\frac{3c_r+r}{2}\right)\omega(n),$$
which gives the desired result.
\end{proof}

According to the previous remark, and in the case of $r=1$ or $2$,
$\eta\big(C^r_{P^-(d)}\big)$ and $\mathsf{D}(C^r_d)$ are known and
easy to compare. Therefore, we can prove Proposition \ref{Majoration
du cross number des groupes de rang 2}.
\begin{proof}[Proof of Proposition \ref{Majoration du cross number des groupes de rang 2}]
Applying Theorem \ref{Proprietes D et Eta}, one can choose $c_1=1$
and $c_2=3$.
\begin{itemize}
\item[$(i)$] For every $n \in \mathbb{N}^*$, one has by Proposition \ref{majoration
qualitative} and Lemma \ref{lemme alpha beta 1erepartie} $(i),(ii)$:
$$\mathsf{k}(C_n) \leq \alpha(n) \leq \alpha_{\omega(n)} \leq \frac{\alpha_9}{9}\omega(n),$$
which proves that one can take $d_1=\alpha_9/9$.

\medskip
\item[$(ii)$] For all $m,n \in \mathbb{N}^*$ with $1 < m |n$, one has by Proposition \ref{majoration
qualitative} applied to $G \simeq C_m \oplus C_n$ and Lemma
\ref{lemme alpha beta 1erepartie} $(i),(iii)$:
$$\mathsf{k}(G) \leq 3\alpha(n)-\beta(n)-\varphi(G,C^2_n) \leq \gamma_{\omega(n)} \leq \frac{\gamma_8}{8}\omega(n),$$
\noindent which proves that one can take $d_2=\gamma_8/8$.
\end{itemize}
\end{proof}

\section{Asymptotical results}\label{Section resultats asymptotiques}
In the present section, we will apply the results obtained in
Section \ref{Section Majoration} in order to prove that
Conjecture \ref{Conjecture Krause} holds asymptotically in the two
directions of Proposition \ref{majoration en fonction des li} and
Proposition \ref{majoration C^r_n bis}.

\begin{proof}[Proof of Proposition \ref{majoration en fonction des li}] First, we have:
$$\lim_{\tiny{\begin{array}{c} P^-(n_r) \rightarrow +\infty \\ C_{n_1} \oplus \dots \oplus C_{n_r} \in
\mathcal{E}_{(l_1,\dots,l_r)}
\end{array}}}\mathsf{k}^*(C_{n_1}
\oplus \dots \oplus C_{n_r})=\displaystyle\sum^r_{i=1}l_i,$$

\noindent and since by the Chinese remainder theorem, every $C_{n_1}
\oplus \dots \oplus C_{n_r}$ in $\mathcal{E}_{(l_1,\dots,l_r)}$ is a
direct summand of $C^r_{n_r}$, we obtain using Lemma \ref{little
cross number d'un sous-groupe} $(ii)$:

\begin{eqnarray*}\mathsf{k}(C_{n_1}\oplus \dots \oplus C_{n_r}) &
\leq & \mathsf{k}\left(C^r_{n_r}\right) -
\mathsf{k}^*\left(C_{\frac{n_r}{n_{r-1}}} \oplus
\dots \oplus C_{\frac{n_r}{n_1}}\right) \\
&=&\mathsf{k}\left(C^r_{n_r}\right)-\displaystyle\sum^{r-1}_{i=1}\mathsf{k}^*\left(C_{\frac{n_r}{n_i}}\right).\end{eqnarray*}

\noindent On the one hand, we have by Lemma \ref{lemme alpha beta
2emepartie} $(ii)$:

$$\displaystyle\limsup_{\tiny{\begin{array}{c} P^-(n_r) \rightarrow +\infty \\
\omega(n_r)=l_r
\end{array}}} \mathsf{k}(C^r_{n_r})
 \leq \limsup_{\tiny{\begin{array}{c} P^-(n_r) \rightarrow +\infty \\
\omega(n_r)=l_r\end{array}}}
c_r\big(\alpha(n_r)-\beta(n_r)\big)+r\beta(n_r)=rl_r,$$

\medskip
\noindent on the other hand, since for all $i \in \llbracket 1,r
\rrbracket$, the equality $\text{gcd}(n_i,n_r/n_i)=1$ implies
$\omega(n_r/n_i)=\omega(n_r)-\omega(n_i)$, we also have:

$$\lim_{\tiny{\begin{array}{c} P^-(n_r) \rightarrow +\infty \\
C_{n_1} \oplus \dots \oplus C_{n_r} \in
\mathcal{E}_{(l_1,\dots,l_r)}
\end{array}}}\displaystyle\sum^{r-1}_{i=1}\mathsf{k}^*\left(C_{\frac{n_r}{n_i}}\right)
=\displaystyle\sum^{r-1}_{i=1}\omega\left(\frac{n_r}{n_i}\right)=\displaystyle\sum^r_{i=1}(l_r-l_i).$$

\medskip
\noindent Finally, we obtain:

$$\displaystyle\sum^r_{i=1}l_i \leq \displaystyle\liminf_{\tiny{\begin{array}{c}
P^-(n_r) \rightarrow +\infty \\
C_{n_1} \oplus \dots \oplus C_{n_r} \in
\mathcal{E}_{(l_1,\dots,l_r)}
\end{array}}} \mathsf{k}(C_{n_1} \oplus \dots \oplus C_{n_r})$$
$$\leq \displaystyle\limsup_{\tiny{\begin{array}{c}
P^-(n_r) \rightarrow +\infty \\
C_{n_1} \oplus \dots \oplus C_{n_r} \in
\mathcal{E}_{(l_1,\dots,l_r)}
\end{array}}} \mathsf{k}(C_{n_1} \oplus \dots \oplus C_{n_r}) \leq rl_r-\displaystyle\sum^r_{i=1}(l_r-l_i)=\displaystyle\sum^r_{i=1}l_i.$$
The corresponding statements for $\mathsf{K}(.)$ and $\mathsf{D}(.)$
are then deduced from Proposition \ref{majoration de K et Dav par
k}.
\end{proof}

Since, as mentioned in Section \ref{Section Introduction},
Proposition \ref{majoration C^r_n} is an immediate corollary of
Proposition \ref{majoration en fonction des li} by specifying
$n_1=\dots=n_r$, we now prove an asymptotical result of an other
type.

\begin{proof}[Proof of Proposition \ref{majoration
C^r_n bis}] First, we have:
$$\displaystyle\lim_{\tiny{\begin{array}{c} \omega(n) \rightarrow +\infty \end{array}}} \frac{\mathsf{k}^*(C^r_n)}{\omega(n)}=r,$$

\medskip
\noindent Moreover, by Proposition \ref{majoration qualitative}
applied to $C^r_n$, we obtain:

$$\frac{\mathsf{k}(C^r_n)}{\omega(n)} \leq
c_r\left(\frac{\alpha(n)-\beta(n)}{\omega(n)}\right)+r\frac{\beta(n)}{\omega(n)},$$

\noindent which implies, by Lemma \ref{lemme alpha beta 2emepartie}
$(i)$, the following inequalities when $\omega(n)$ tends to infinity:

$$r=\displaystyle\lim_{\tiny{\begin{array}{c} \omega(n) \rightarrow +\infty \end{array}}} \frac{\mathsf{k}^*(C^r_n)}{\omega(n)}
\leq \displaystyle\lim_{\tiny{\begin{array}{c} \omega(n) \rightarrow
+\infty \end{array}}} \frac{\mathsf{k}(C^r_n)}{\omega(n)} \leq r.$$

\medskip
\noindent The result for $\mathsf{K}(.)$ is then deduced from
Proposition \ref{majoration de K et Dav par k} $(i)$.
\end{proof}

Each of the two previous asymptotical results admits a corollary
which may appear more general at first sight. So as to state the
first one, we will use the following notation, which recalls the one
used for the sets $\mathcal{E}_{(l_1,\dots,l_r)}$. For every $r,l
\in \mathbb{N}^*$, we set:
$$\mathcal{E}_{r,l}=\{G \text{ finite Abelian group } | \text{ }
\text{r}(G)=r, \text{ } \omega(\exp(G))=l\}.$$

\noindent With this notation, we obtain the following corollary.
\newpage
\begin{corollary}\label{majoration en fonction de l} For all integers $r,l \in \mathbb{N}^*$ the three following statements hold:
\begin{itemize}
\item[$(i)$] $$\displaystyle\limsup_{\tiny{\begin{array}{c} P^-(\exp(G)) \rightarrow +\infty \\ G \in \mathcal{E}_{r,l} \end{array}}} \mathsf{k}(G) \leq rl,$$

\item[$(ii)$] $$\displaystyle\limsup_{\tiny{\begin{array}{c} P^-(\exp(G)) \rightarrow +\infty \\ G \in \mathcal{E}_{r,l} \end{array}}} \mathsf{K}(G) \leq rl,$$

\item[$(iii)$]$$\displaystyle\limsup_{\tiny{\begin{array}{c} P^-(\exp(G)) \rightarrow +\infty \\ G \in \mathcal{E}_{r,l} \end{array}}}
\frac{\mathsf{D}(G)}{\exp(G)} \leq rl.$$
\end{itemize}
\end{corollary}

\begin{proof} Every $G$ in $\mathcal{E}_{r,l}$ can be injected in the group $H \simeq C^r_{\exp(G)}$. Therefore, using
Lemma \ref{little cross number d'un sous-groupe}, we obtain
$\mathsf{k}(G) \leq \mathsf{k}(H)$ and the desired result follows
from Proposition \ref{majoration C^r_n}, applied to the group $H$.
The corresponding statements for $\mathsf{K}(.)$ and $\mathsf{D}(.)$
are then deduced from Proposition \ref{majoration de K et Dav par
k}.
\end{proof}

\begin{corollary}\label{} For all integers $r \in \mathbb{N}^*$, the two following statements hold:
\begin{itemize}
\item[$(i)$] $$\displaystyle\limsup_{\tiny{\begin{array}{c} \omega(\exp(G))
\rightarrow +\infty \\ r(G) \leq r\end{array}}}
\frac{\mathsf{k}(G)}{\omega(\exp(G))} \leq r,$$

\item[$(ii)$] $$\displaystyle\limsup_{\tiny{\begin{array}{c} \omega(\exp(G))
\rightarrow +\infty \\ r(G) \leq r\end{array}}}
\frac{\mathsf{K}(G)}{\omega(\exp(G))} \leq r.$$
\end{itemize}
\end{corollary}

\begin{proof} Every $G$ with rank $r(G) \leq r$ can be injected in the group $H \simeq C^r_{\exp(G)}$. Since we
have $\mathsf{k}(G) \leq \mathsf{k}(H)$ by Lemma \ref{little cross
number d'un sous-groupe}, the result follows from Proposition
\ref{majoration C^r_n bis}, applied to the group $H$. The statement
for $\mathsf{K}(.)$ is then deduced from Proposition \ref{majoration
de K et Dav par k} $(i)$.
\end{proof}

\section*{Acknowledgments}
I am grateful to my Ph.D. advisor Alain Plagne for his help during
the preparation of this article. I would like also to thank Wolfgang
Schmid for his remarks on a preliminary version of this paper.


\end{document}